\documentclass{lms}

\usepackage{graphicx}
\usepackage{amsmath,esint}
\usepackage{amsfonts}
\usepackage{amssymb}
\usepackage{color}
\usepackage{lmodern}
\usepackage{dsfont}
\usepackage{bbm}
\usepackage[draft]{todonotes}

\usepackage[normalem]{ulem}
\usepackage{amsmath}
\usepackage{appendix}

\def\be{\begin{eqnarray}}
\def\ee{\end{eqnarray}}
\def\b*{\begin{eqnarray*}}
\def\e*{\end{eqnarray*}}

\def \E{\mathbb{E}}

\def \N{\mathbb{N}}

\def \R{\mathbb{R}}
\def \S{\mathbb{S}}

\def\Cc{{\mathcal C}}
\def\Dc{{\mathcal D}}
\def\Ec{{\mathcal E}}

\def\Lc{{\mathcal L}}

\def\Oc{{\mathcal O}}
\def\Pc{{\mathcal P}}

\def\Tc{{\mathcal T}}

\def\Wc{{\mathcal W}}

\def \eps {\varepsilon}

\def \sb {\subset}

\def \bu{\mathsf{u}}
\def \bv{\mathsf{v}}
\def \bx{\mathsf{x}}
\def \by{\mathsf{y}}

\def \bz{\mathsf{z}}
\def \sW{\msout{\mathcal{W}}}
\def \mW{\overline{\mathcal{W}}}

\def \lip{\mathsf{lip}}
\def \Lee{\mathsf{Lip}}

\def \hmu{\hat{\mu}}

\def \0{\mathsf{0}}
\def \1{\mathsf{1}}

\newcommand{\stkout}[1]{\ifmmode\mbox{\sout{\ensuremath{#1}}}\else\sout{#1}\fi}

\newcommand{\msout}[1]{\text{\sout{\ensuremath{#1}}}}

\newtheorem{theorem}{Theorem}[section] 
\newtheorem{lemma}{Lemma}[section]    

\newtheorem{proposition}{Proposition}[section]
\newtheorem{definition}{Definition}[section]
\newtheorem{remark}{Remark}[section]

\title[From Wasserstein to max-sliced Wasserstein]
 {Strong equivalence between metrics of Wasserstein type} 

\author{\sc Erhan Bayraktar and Gaoyue Guo}

\classno{: }

\extraline{{\bf Keywords : } optimal transport, Wasserstein metric, max-sliced Wasserstein metric, duality 

\vspace{2mm}

\noindent We thank the anonymous referee for the thorough review and highly appreciate the comments and
suggestions, which significantly contributed to improving the quality of the publication. The second author is grateful for the support of CentraleSup\'elec, and in addition, to the University of Michigan and AMS Simons Travel
Grant.}

\begin{document}
\maketitle

\begin{abstract}
The sliced Wasserstein metric $\sW_p$ and more recently max-sliced Wasserstein metric  $\mW_p$ have attracted abundant attention in data sciences and machine learning due to their advantages to tackle the curse of dimensionality, see e.g. \cite{10.1007/978-3-642-24785-9_37}, \cite{DHSPSKZFS}. 
A question of particular importance is the strong equivalence between these projected Wasserstein metrics and the (classical) Wasserstein metric $\Wc_p$. Recently, Paty and Cuturi have proved in \cite{PC} the strong equivalence of $\mW_2$ and $\Wc_2$.  We show that the strong equivalence also holds for $p=1$, while the sliced Wasserstein metric does not share this nice property.
\end{abstract}

\section{Introduction}\label{sec:intro}

\noindent The Wasserstein metric arising in the optimal transport theory  forms a distance function between probability measures. In  mathematical language, the Wasserstein distance of order $p\ge 1$ between  probability measures $\mu$ and $\nu$ on $\R^d$ is defined as
\b* 
\Wc_p(\mu,\nu) &:=& \inf_{\gamma\in\Gamma(\mu,\nu)}  \left(\int_{\R^d\times\R^d} |\bx-\by|^p\gamma(d\bx,d\by)\right)^{1/p},
\e* 
where $\Gamma(\mu,\nu)$ is the set of probability  measures $\gamma$ on $\R^d\times\R^d$ having the  marginal distributions $\mu$ and $\nu$. Theoretical advances in the last fifty years characterize existence, uniqueness, representation and smoothness properties of optimizers for $\Wc_p(\mu,\nu)$ under different settings and compute $\Wc_p(\mu,\nu)$ by adopting tools and methods in PDE, linear programming and computational geometry, see e.g. \cite{RR1}, \cite{Villani}, and applications are various  throughout most of the applied sciences including economics, geography and biomedical sciences, see e.g. \cite{RR2},  \cite{Santambrogio}. Recently, it has attracted abundant attention in data sciences and machine learning due to its  theoretical properties and applications on many domains, see e.g. Wasserstein GAN (Generative Adversarial Network) in \cite{ACB}. While the Wasserstein metric brings new perspectives and principled ways to formalize problems, the related methods usually suffer from high computational complexity as  evaluating Wasserstein distance for $d\gg 1$ is numerically intractable in general. This important computational burden is a major limiting factor in the application of Wasserstein metric to large-scale data analysis.
An appealing path to overcome the curse of dimensionality is the recently introduced sliced Wasserstein metric $\sW_p$,  which is based on the average Wasserstein distance of the projections of two distributions, see e.g. \cite{10.1007/978-3-642-24785-9_37}, \cite{Bonnotte},  \cite{Santambrogio}.
Very recently, in order to reduce the projection complexity of the sliced Wasserstein, \cite{DHSPSKZFS}
introduced the so-called max-Wasserstein metrics, which we will denote by $\mW_p$, as a fix. The same paper also points out that both of these projected versions of the Wasserstein distance enjoy the so-called generalizability over the Wasserstein metric. For further recent results please refer to \cite{KNSBR}, \cite{PC}.

Paty and Cuturi showed in \cite{PC} that the max-sliced distance $\mW_2$ is strongly equivalent to $\Wc_2$. This paper aims to prove  this result  for $p=1$.  The proof of our result is based on  the dual formulation of $\Wc_1$, constructing a tailor-made topology $\tau$ on the space of Lipschitz functions on $\R^d$, and some functional analytic arguments. This is reminiscent of the universal approximation result in e.g. \cite{Cybenko}, i.e. any arbitrary Lipschitz function can be recovered from functional evaluation of projections. Although in the same spirit, our result here is different. In Theorem \ref{prop:inclusion}, we prove there exists $C_d>0$ such that the collection of functions as below 
\be\label{eq:approx}
\sum_{1\le k\le n} a_k f_k(\bv_k\cdot \bx)
\ee
is dense, endowed with $\tau$, in the absorbing and convex set of $1-$Lipschitz functions $\Lee_1(\R^d)$, where $n\in\N$, $a_k\ge 0$, $\bv_k\in\S^{d-1}$,    $f_k\in \Lee_1(\R)$ for $1\le k\le n$ and $\sum_{1\le k\le n}a_k\le C_d$. Roughly speaking, any $1-$Lipschitz function on $\R^d$ can be approximated by a sequence of $C_d-$Lipschitz functions of form \eqref{eq:approx}.
 
We show further that the strong equivalence is not shared by the sliced Wasserstein metric using the recent results of \cite{AGT}, hence promoting the max-sliced metric over the sliced one.

The structure of the rest of the paper is simple. In the next section, after introducing some preliminaries we will give our main results in Theorem~\ref{thm:equivalence}.  Section~\ref{sec:proof}, on the other hand, is devoted to the proof of these results. A  technical lemma is presented in the Appendix.

\section{Main Results}\label{sec:pre}

\subsection{Preliminaries on the Wasserstein Metric}
\noindent We start by reviewing  the preliminary concepts and formulations needed to introduce the main results. For $p\ge 1$, let  $\Pc_p(\R^d)$ be the set of probability measures on $\R^d$ admitting finite $p^{\rm th}$ moment, i.e.   $\mu\in \Pc_p(\R^d)$ if and only if  
\be\label{def:moment} 
M_p(\mu) := \left(\int_{\R^d}|\bx|^p\mu(d\bx)\right)^{1/p} < \infty.
\ee    
For $\mu,\nu\in\Pc_p(\R^d)$,  denote by $\Gamma(\mu,\nu)\sb\Pc_p(\R^d\times\R^d)$ the collection of probability measures $\gamma$ on $\R^d\times\R^d$, also known as couplings or transport plans, such that
\b* 
\gamma\big[E\times \R^d\big] = \mu[E] \mbox{ and } \gamma\big[\R^d\times E\big] = \nu[E],\quad  \mbox{for all measurable sets } E\sb\R^d.
\e* 
 The Wasserstein metric of order $p$ is a distance function $\Wc_p:\Pc_p(\R^d)\times \Pc_p(\R^d)\to\R_+$ defined by 
\be 
\Wc_p(\mu,\nu) &:=& \inf_{\gamma\in\Gamma(\mu,\nu)}  \left(\int_{\R^d\times\R^d} |\bx-\by|^p\gamma(d\bx,d\by)\right)^{1/p}, \quad \mbox{for all } \mu, \nu\in\Pc_p(\R^d).
\ee 
It is known that $\Pc_p(\R^d)$ endowed with $\Wc_p$ is a Polish space, i.e. separable completely metrizable  space, see e.g. Theorem 6.18 of \cite{Villani}. In particular, an explicit  expression of $\Wc_p(\mu,\nu)$ is given for  $d=1$:
\be\label{eq:formula}
\Wc_p(\mu,\nu) &= & \left(\int_{0}^{1} \big|F_{\mu}^{-1}(t) - F_{\nu}^{-1}(t)\big|^p dt\right)^{1/p},
\ee 
where $F_{\mu}^{-1}(t):=\inf\{x\in\R: \mu[(-\infty, x]]>t\}$ and $F_{\nu}^{-1}(t):=\inf\{x\in\R: \nu[(-\infty, x]]>t\}$, see e.g. Chapter 2 of \cite{Santambrogio}. We note also that $\Wc_p$ is depending on $d$, i.e. $\Wc_p\equiv \Wc_{p,d}$. Nevertheless, we will not emphasize this dependency in the rest and write simply $\Wc_p$ without any danger of confusion.

\subsection{Projected Wasserstein Metrics}
\noindent While approaches based on the Wasserstein metric have been successful in several complex  tasks, estimating the Wasserstein distance often  suffers  from the curse of dimensionality from the complexity/algorithmic perspective. To tackle the issue of complexity, a sliced version of the Wasserstein distance was studied and employed, which only requires estimating distances of projected uni-dimensional  distributions and is, therefore, more efficient, see e.g. \cite{Bonnotte}, \cite{DHSPSKZFS},  \cite{KNSBR}. 

Let $\S^{d-1}\sb\R^d$ be the unit sphere, i.e. $\S^{d-1}:=\{\bv\in\R^d : |\bv|=1\}$. For $\mu\in\Pc_p(\R^d)$ and $\bv\in\S^{d-1}$, set $\mu_{\bv}:= \mu\circ \bv_*^{-1}$ to be the image measure of $\mu$ by $\bv_*$, where $\bv_*:\R^d\to\R$ is the map defined by $\bv_*(\bx):=\bv\cdot \bx$. Then $\mu_{\bv}\in\Pc_p(\R)$ since 
\b* 
\int_{\R}|x|^p\mu_{\bv}(dx) = \int_{\R^d}|\bv\cdot \bx|^p\mu(d\bx) \le \int_{\R^d}|\bx|^p\mu(d\bx).
\e* 
 Hence,  we may define the sliced Wasserstein metric $\sW_p$ and max-sliced Wasserstein metric $\mW_p$ as follows.
\begin{definition}
For $\mu,\nu\in\Pc_p(\R^d)$, set 
\b* 
\sW_p(\mu,\nu) :=  \left(\oint_{\S^{d-1}} \Wc_p(\mu_{\bv},\nu_{\bv})^pd\bv\right)^{1/p} &\quad\mbox{and}\quad& 
\mW_p(\mu,\nu) :=  \sup_{\bv\in\S^{d-1}} \Wc_p(\mu_{\bv},\nu_{\bv}),
\e* 
where $\oint$ denotes the surface integral over $\S^{d-1}$.
\end{definition}
For $d=1$, one has by definition 
 $\Wc_p(\mu,\nu)=2^{-1/p}\sW_p(\mu,\nu) =\mW_p(\mu,\nu)$. Further, Proposition \ref{prop:lip} ensures that $\sW_p, \mW_p$ are well defined metrics on $\Pc_p(\R^d)$.
\begin{proposition}\label{prop:lip}
{\rm (i)} Fix $\mu, \nu\in\Pc_p(\R^d)$. The maps   $\S^{d-1}\ni\bv\mapsto \mu_{\bv}\in \Pc_p(\R)$ and $\S^{d-1}\ni\bv\mapsto\Wc_p(\mu_{\bv},\nu_{\bv})\in\R_+$ are both Lipschitz, with respectively Lipschitz constants $M_p(\mu)$ and  $M_{p}(\mu)+M_{p}(\nu)$. In particular, the supremum over $\S^{d-1}$ can be attained and   $\mW_p(\mu,\nu)=  \max_{\bv\in\S^{d-1}} \Wc_p(\mu_{\bv},\nu_{\bv})$.

\vspace{1mm}

\noindent {\rm (ii)} $\sW_p$ and  $\mW_p$ form two distance functions on $\Pc_p(\R^d)$.
\end{proposition}
\begin{proof}
{\rm (i)} For all $\bu,\bv\in\S^{d-1}$, let  $\bu_*\otimes \bv_*:\R^d\times\R^d\to\R^2$ be the map defined by $\bu_*\otimes \bv_*(\bx,\by):=(\bu\cdot \bx, \bv\cdot \by)$. Taking  $\gamma(d\bx,d\by):=\mu(dx)\otimes \delta_{\bx}(d\by)
\in\Gamma(\mu,\mu)$, one has $\gamma_{\bu,\bv}:=\gamma \circ (\bu_*\otimes \bv_*)^{-1}\in \Gamma(\mu_{\bu},\mu_{\bv})$ and thus
\b* 
\Wc_p(\mu_{\bu},\mu_{\bv}) \le \left(\int_{\R^2}|x-y|^p\gamma_{\bu,\bv}(dx,dy)\right)^{1/p} = \left(\int_{\R^d\times\R^d}|\bu\cdot \bx-\bv\cdot \by|^p\gamma(dx,dy)\right)^{1/p} \le  |\bu-\bv| M_{p}(\mu).
\e* 
Further, the triangle inequality yields
$\big|\Wc_p(\mu_{\bu},\nu_{\bu}) - \Wc_p(\mu_{\bv},\nu_{\bv}) \big| \le \big|\Wc_p(\mu_{\bu},\mu_{\bv})  \big| +  \big|\Wc_p(\nu_{\bu},\nu_{\bv}) \big| \le  |\bu-\bv|\big(M_{p}(\mu)+M_{p}(\nu)\big)$, 
which yields the Lipschitz continuity and further $\mW_p(\mu,\nu)=  \max_{\bv\in\S^{d-1}} \Wc_p(\mu_{\bv},\nu_{\bv})$.

\vspace{1mm}

\noindent {\rm (ii)}  The symmetry and sub-additivity are trivial by definition. Denote by $A_d$ the surface area of $\S^{d-1}$, i.e. 
\be \label{def:area} 
A_d:=\oint_{\S^{d-1}}d\bv = \frac{2\pi^{d/2}}{\Gamma(d/2)},
\ee 
where $\Gamma:\R\to\R$ is the Gamma function given as
\b* 
\Gamma(x) &:=& \int_0^{\infty} t^{x-1}e^{-t}dt.
\e* 
Then one has $\sW_p(\mu,\nu)\le A_d^{1/p}\mW_p(\mu,\nu)$ for all $\mu,\nu\in\Pc_p(\R^d)$ and it suffices to show the identity of indiscernibles for $\sW_p$. Let $\sW_p(\mu,\nu)=0$, which implies by {\rm (i)} that $\mu_{\bv}=\nu_{\bv}$ for all $\bv\in\S^{d-1}$. Consider the characteristic functions of $\mu, \nu$ defined by
\b* 
    \tilde{\mu}(\bz) := \int_{\R^d} e^{i\bz\cdot \bx}\mu(d\bx) &\quad\mbox{and}\quad&
\tilde{\nu}(\bz):= \int_{\R^d} e^{i\bz\cdot \bx}\nu(d\bx).
\e* 
With $r:=|\bz|$ and $\bv:=\bz/r$, it follows that
\b* 
\int_{\R^d} e^{i\bz\cdot \bx}\mu(d\bx)  = \int_{\R} e^{irx}\mu_{\bv}(dx) = \int_{\R} e^{ir x}\nu_{\bv}(dx)=\int_{\R^d} e^{i\bz\cdot \bx}\nu(d\bx),
\e* 
which implies $\tilde{\mu}(\bz)=\tilde{\nu}(\bz)$ for all $\bz\in\R^d$ and finally  $\mu=\nu$.
\end{proof}

\subsection{Main Results}

\noindent Given the active theoretical interest of Wasserstein metric, as well as its importance for applications in practice, the investigation of $\sW_p$ and $\mW_p$ is gaining popularity in machine learning, with several applications to data sciences. A question of particular importance is the equivalence between $\sW_p, \mW_p$ and $\Wc_p$. Recently, Paty and Cuturi proved in \cite{PC} the strong equivalence of $\mW_2$ and $\Wc_2$. Namely,
\b* 
\mW_2(\mu,\nu)\le \Wc_2(\mu,\nu)\le \sqrt{d}\mW_2(\mu,\nu),\quad \mbox{for all } \mu,\nu\in\Pc_2(\R^d).
\e* 
In this paper, we show the (topological) equivalence between $\sW_p, \mW_p$ and $\Wc_p$ as well as the strong equivalence between $\mW_1$ and $\Wc_1$, which are summarized in Theorem  \ref{thm:equivalence} below.
\begin{theorem}\label{thm:equivalence}
{\rm (i)} $\sW_p$, $\mW_p$ and $\Wc_p$ are equivalent for all $p\ge 1$, i.e. 
\b* 
\lim_{n\to\infty}\sW_p(\mu^n,\mu)=0 \quad \Longleftrightarrow \quad \lim_{n\to\infty}\mW_p(\mu^n,\mu)=0 \quad \Longleftrightarrow \quad \lim_{n\to\infty}\Wc_p(\mu^n,\mu)=0
\e* 
for any sequence $(\mu^n)_{n\ge 1}\sb\Pc_p(\R^d)$ and $\mu\in\Pc_p(\R^d)$. 

\vspace{1mm}

\noindent {\rm (ii)} $\mW_1$ and $\Wc_1$ are strongly equivalent for all $d\ge 1$, i.e. there exists $C_d\ge 1$ such that
\be\label{eq:equiv}
\mW_1(\mu,\nu)\le \Wc_1(\mu,\nu)\le C_d\mW_1(\mu,\nu),\quad \mbox{for all } \mu,\nu\in\Pc_1(\R^d).
\ee  
{\rm (iii)} $\sW_1$ and $\Wc_1$ are not strongly equivalent for all $d\ge 2$.
\end{theorem}

 \begin{remark}\footnote{This observation is suggested the anonymous referee and provides a tractable schema to approximate $C_d^*$ by solving an optimization problem over the compact subset $\Pc_1(B_1)$.}
For every $R>0$, denote by $\Pc(B_R)\subset \Pc_1(\R^d)$ the subset of probability measures supported in $B_R:=\{\bx\in\R^d: |\bx|\le R\}$. Then one must have $0<C^*_d\le C_d$ such that  
\b*
\Wc_1(\mu,\nu) \le C^*_d\mW_1(\mu,\nu),\quad  \mbox{for all }  \mu,\nu\in\Pc(B_1).
\e*
 Let $\mu\in\Pc(B_R)$ for some $R>0$ and $X\sim \mu$. Denote by $\mu^R$ the distribution of $X/R$, then clearly $\mu^R\in\Pc(B_1)$. Further, for $\mu,\nu\in\Pc(B_R)$, one has $
\Wc_1(\mu,\nu) = R\Wc_1(\mu^R,\nu^R)$  and $\mW_1(\mu,\nu) = R\mW_1(\mu^R,\nu^R)$, 
which implies 
\b*
\Wc_1(\mu,\nu) \le C^*_d\mW_1(\mu,\nu), \quad \mbox{for all } \mu,\nu\in\bigcup_{R>0}\Pc(B_R).
\e*
and further $\Wc_1 \le C^*_d\mW_1$ on $\Pc_1(\R^d)$ as $\bigcup_{R>0}\Pc(B_R)$ is dense in $\Pc_1(\R^d)$ under $\Wc_1$ and $\mW_1$.
\end{remark}

\section{Proof of the Main Results}\label{sec:proof}

\subsection{Proof of Theorem \ref{thm:equivalence} {\rm (i)} }\label{ssec:proof1}
\noindent Given  $\bu,\bv\in\S^{d-1}$, let  $\bu_*\otimes \bv_*:\R^d\times\R^d\to\R^2$  be the map defined by 
\be\label{def:image}
\bu_*\otimes \bv_*(\bx,\by) &:=& (\bu\cdot \bx, \bv\cdot \by).
\ee
For $\mu,\nu\in\Pc_p(\R^d)$, let  $\gamma\in\Gamma(\mu,\nu)$ be an optimizer for $\Wc_p(\mu,\nu)$. Then, by definition, $\gamma_{\bu,\bv}:=\gamma\circ (\bu_*\otimes \bv_*)^{-1}\in\Gamma(\mu_{\bu},\nu_{\bv})$. Taking in particular $\bu=\bv$, one has 
\b* 
\Wc_p(\mu_{\bv},\nu_{\bv})^p \le \int_{\R^2}|x-y|^p\gamma_{\bv,\bv}(dx,dy) = \int_{\R^d\times \R^d}|\bv\cdot \bx-\bv\cdot \by|^p\gamma(d\bx,d\by)\le \Wc_p(\mu,\nu)^p,
\e* 
which yields the trivial inequality as follows:
\be\label{ineq1}
\frac{1}{A_d^{1/p}}\sW_p(\mu,\nu)\le \mW_p(\mu,\nu) \le \Wc_p(\mu,\nu), 
\ee
where $A_d$ is defined in \eqref{def:area}. Therefore, it remains to show the equivalence between $\sW_p$ and $\Wc_p$.

For the sake of presentation, we use the following notation for the fact that random variable $X$ is distributed according to probability measure $\mu$:
\b* 
X\sim \mu &\quad \mbox{or}\quad & \Lc(X) = \mu,
\e* 
where $\Lc(X)$ denotes the law of $X$.
\begin{proof}[of Theorem \ref{thm:equivalence} {\rm (i)}]
It suffices to prove that $\lim_{n\to\infty}\sW_p(\mu^n,\mu)=0$ implies $\lim_{n\to\infty}\Wc_p(\mu^n,\mu)=0$. We proceed as follows. 

\vspace{1mm}

\noindent \emph{Step 1.} For each $n\ge 1$, let $X^n$ be a random variable with $X^n\sim \mu^n$, and by definition, $\bv\cdot X^n\sim \mu^n_{\bv}$ holds for all  $\bv\in\S^{d-1}$. As
\b* 
\lim_{n\to\infty} \oint_{\S^{d-1}} \Wc_p(\mu^n_{\bv},\mu_{\bv})^pd\bv = \lim_{n\to\infty}  \sW_p(\mu^n_{\bv},\mu_{\bv})^p =0, 
\e* 
the functions $\S^{d-1}\ni\bv\mapsto  \Wc_p(\mu^n_{\bv},\mu_{\bv})^p\in\R_+$ converge in measure to zero.  Namely, for almost every $\bv\in\S^{d-1}$,   $\lim_{n\to\infty}\Wc_p(\mu_{\bv}^n,\mu_{\bv})=0$. We can further conclude that the sequence $\big(|\bv\cdot X^n|^p\big)_{n\ge 1}$ is uniformly integrable. Pick a finite set $\{\bv_1,\ldots, \bv_{I}\}\sb\S^{d-1}$ such that
\b* 
\lim_{n\to\infty}\Wc_p(\mu_{\bv_i}^n,\mu_{\bv_i})=0 \mbox{ for } 1\le i\le I &\quad \mbox{and}\quad& |\bx|/2\le \max_{1\le i\le I}|\bv_i\cdot \bx| \mbox{ for all } \bx\in\R^d.
\e* 
Hence, $|X^n|^p \le 2^p\sum_{i=1}^I|\bv_i\cdot X^n|^p$ yields the uniform integrability of $\big(|X^n|^p\big)_{n\ge 1}$ and in particular 
\b* 
\sup_{n\ge 1} M_p(\mu^n) = \sup_{n\ge 1} \left(\int_{\R^d}|\bx|^p\mu^n(d\bx)\right)^{1/p} = \sup_{n\ge 1} \left(\E\big[|X^n|^p\big]\right)^{1/p} =: C < \infty.
\e* 
In view of the proof of Proposition \ref{prop:lip}, the maps $\S^{d-1}\ni\bv\mapsto  \Wc_p(\mu^n_{\bv},\mu_{\bv})^p\in\R_+$ are equi-Lipschitz with a uniform Lipschitz constant  $C+M_p(\mu)$, and thus 
\be \label{dircon}
\lim_{n\to\infty}\Wc_p(\mu_{\bv}^n,\mu_{\bv})=0,\quad \mbox{for all } \bv\in\S^{d-1}.
\ee  
\emph{Step 2.} Consider the characteristic function of $\mu$ given by
\b* 
\tilde{\mu}(\bz) &:=& \int_{\R^d} e^{i\bz\cdot \bx}\mu(d\bx),\quad \mbox{for all } \bz\in\R^d. 
\e* 
Define similarly $\tilde{\mu}^{n}$ for all $n\ge 1$. For every $\bz\in\R^d$, with $r:=|\bz|$ and $\bv:=\bz/r$, it holds
\b* 
\lim_{n\to\infty}\int_{\R^d} e^{i\bz\cdot \bx}\mu^{n}(d\bx)  = \lim_{n\to\infty}\int_{\R} e^{irx}\mu^{n}_{\bv}(dx) = \int_{\R} e^{ir x}\mu_{\bv}(dx)=\int_{\R^d} e^{i\bz\cdot \bx}\mu(d\bx),
\e* 
where the second equality follows from \eqref{dircon}. We conclude thus $(\mu^n)_{n\ge 1}$ converges weakly to $\mu$.

\vspace{1mm}

\noindent \emph{Step 3.} Using the Skorokhod representation theorem, we may assume without loss of generality that the sequence $(X^n)_{n\ge 1}$ converges almost surely. Denote by $X\sim\mu$ its limit. Combining with the uniform integrability of $\big(|X^n|^p\big)_{n\ge 1}$, one has 
\b* 
\lim_{n\to\infty}\Wc_p(\mu^n,\mu)^p \le \lim_{n\to\infty}\E\big[|X^n-X|^p\big] = 0, 
\e* 
which fulfils the proof. 
\end{proof}

\subsection{Proof of Theorem \ref{thm:equivalence} {\rm (ii)} }\label{ssec:proof2}
\noindent Our proof is based on the dual formulation of $\Wc_1$ and inspired by the proof of the universal approximation theorem. Let $L^0(\R^d)$ be space of Lebesgue measurable functions $f:\R^d\to\R$ and $\Lee(\R^d)\sb L^0(\R^d)$ be the subspace consisting of Lipschitz functions vanishing at the origin, i.e. $f\in\Lee(\R^d)$ if and only if $f(\0)=0$ and 
\b* 
\|f\|_{\lip} := \sup_{\bx\neq \by\in\R^d} \frac{|f(\bx)-f(\by)|}{|\bx-\by|}<\infty.
\e* 
For each $L>0$,  denote by $\Lee_L(\R^d)\sb \Lee(\R^d)$ the subset of functions $f$ with $\|f\|_{\lip}\le L$.
Then it follows by Kantorovich's duality that, see e.g. Remark 6.5 in \cite{Villani}, 
\be\label{def:dual}
\Wc_1(\mu,\nu) &=& \sup_{f\in\Lee_1(\R^d)} \left\{\int_{\R^d}f(\bx)\mu(d\bx) - \int_{\R^d}f(\bx)\nu(d\bx)\right\},\quad \mbox{for all } \mu,\nu\in\Pc_1(\R^d).
\ee 
In what follows, \eqref{def:dual} will be used in the proof of Theorem \ref{thm:equivalence}. It is known from \cite{Johnson}, $\|\cdot\|_{\lip}$ defines a norm on $\Lee(\R^d)$ and $\big(\Lee(\R^d), \|\cdot\|_{\lip}\big)$ is a Banach space. Next we endow $\Lee(\R^d)$ with an alternative topology.  
Set
\b* 
L^1(\R^d)&:=&\left\{f\in L^0(\R^d):~ \|f\|_1:=\int_{\R^d}|f(\bx)|d\bx<\infty\right\} \\
L^{\infty}(\R^d)&:=&\left\{f\in L^0(\R^d) :~ \|f\|_{\infty}:={\rm ess} \sup_{\bx\in \R^d}|f(\bx)|<\infty\right\},
\e* 
then $\big(L^1(\R^d), \|\cdot\|_1\big)$ and $\big(L^{\infty}(\R^d), \|\cdot\|_{\infty}\big)$ are both  Banach spaces. Denote further by  $L^1(\R^d)^d$ (resp. $L^{\infty}(\R^d)^d$) the $d-$product of $L^1(\R^d)$ (resp. $L^{\infty}(\R^d)^d$). Note in particular that every $f\in \Lee(\R^d)$ is a.e. differentiable and $\nabla f\in  L^{\infty}(\R^d)^d$ with $\|\nabla f\|_{\infty}=\|f\|_{\lip}$, see e.g. Exercise 6.14 of \cite{Heinonen}. Finally we define $L^1\big(\R^d, (1+|\bx|)d\bx\big)\sb L^1(\R^d)$ by
\b* 
L^1\big(\R^d,(1+|\bx|)d\bx\big)&:=&\left\{f\in L^1(\R^d) :~  \int_{\R^d}(1+|\bx|)|f(\bx)|d\bx<\infty\right\}. 
\e* 
Now we are ready to introduce the topology, denoted by $\tau$, on $\Lee(\R^d)$ for which $\big(\Lee(\R^d), \tau\big)$ is a locally convex space. Let $\tau^o$ be a collection of subsets  $\Oc_{(u,w)}(f;\eps)\subset \Lee(\R^d)$ defined as follows:
\b* 
\Oc_{(u,w)}(f;\eps) 
&:=&\left\{g\in \Lee(\R^d):~ \left|\int_{\R^d}\big[(g-f)(\bx) u(\bx)+ \nabla (g-f)(\bx)\cdot w(\bx)\big]d\bx \right|<\eps\right\}, 
\e* 
where $f\in \Lee(\R^d)$, $\eps>0$ and $(u,w)\in L^1\big(\R^d, (1+|\bx|)d\bx\big)\times  L^1(\R^d)^d$. Then we define $\tau$ to be the topology generated by $\tau^o$.  $\big(\Lee(\R^d),\tau\big)$ is a locally convex space, as the origin has a local base of absolutely convex absorbent sets,  see e.g. Proposition 1.15 in Chapter 4 of \cite{Conway}. Under the topology $\tau$, $(f_n)_{n\ge 1}\sb\Lee(\R^d)$ converges to $f\in \Lee(\R^d)$  if and only if
\b* 
\lim_{n\to\infty}\int_{\R^d}\big[(f_n-f)(\bx) u(\bx)+ \nabla (f_n-f)(\bx)\cdot w(\bx)\big]d\bx &=&0
\e* 
holds for all $(u, w)\in L^1\big(\R^d, (1+|\bx|)d\bx\big)\times  L^1(\R^d)^d$.  \begin{remark}
{\rm (i)} In view of \cite{GL}, $\Lee(\R^d)$ is is the dual space of some closed quotient space of $L^1(\R^d)^d$ and $\tau$ turns to be the 
weak* topology. Hence, $\big(\Lee(\R^d),\tau\big)$ is not metrizable and thus not a Fr\'echet space as this quotient space is infinite dimensional, see e.g. Chapter 5 of \cite{Conway}. However, Proposition \ref{prop:closed} shows that $\tau$, restricted to $\Lee_L(\R^d)\sb \Lee(\R^d)$ for any $L>0$, is metrizable.

\vspace{1mm}

\noindent {\rm (ii)} For any sequence $(g_n)_{n\ge 1}\subset \Lee(\R^d)$ with $\lim_{n\to\infty}\|g_n\|_{\lip}=0$, one has for any $(u,w)\in L^1\big(\R^d, (1+|\bx|)d\bx\big)\times  L^1(\R^d)^d$
\b* 
\lim_{n\to\infty} \left|\int_{\mathbb R^d} g_n(\bx)u(\bx)d\bx\right| = \lim_{n\to\infty} \left|\int_{\mathbb R^d} \left(\int_0^1 \bx\cdot \nabla g_n(t\bx) dt\right) u(\bx)d\bx\right| \le \lim_{n\to\infty} \| g_n\|_{\lip}\int_{\mathbb R^d}  |\bx| u(\bx)d\bx=0
\e* 
and 
\b* 
\lim_{n\to\infty} \left|\int_{\mathbb R^d} \nabla g_n(\bx)\cdot w(\bx)d\bx\right| \le  \lim_{n\to\infty} \| g_n\|_{\lip}\int_{\mathbb R^d}  |w(\bx)|d\bx = 0.
\e* 
Therefore, the topology under $\|\cdot\|_{\lip}$ is strictly stronger than $\tau$. 
\end{remark}

 The lemma below characterizes the space of $\tau-$continuous linear functions on $\Lee(\R^d)$. 
\begin{lemma}\label{lem:dual}
Any $\tau-$continuous linear  functions $T: \Lee(\R^d)\to\R$ must be of the form 
\b* 
T(f) &=& \int_{\R^d}\big[f(\bx) u(\bx)+ \nabla f(\bx)\cdot w(\bx)\big]d\bx,
\e* 
for some $(u, w)\in L^1\big(\R^d, (1+|\bx|)d\bx\big)\times  L^1(\R^d)^d$.
\end{lemma}
Let $\Cc\sb \Lee_1(\R^d)$ be the subset of functions $f$ of the form:
\b* 
f(\bx) &=& \sum_{1\le k\le n} a_k f_k(\bv_k\cdot \bx),
\e* 
where $n\in\N$, $a_k\in\R_+$, $\bv_k\in\S^{d-1}$,  $f_k\in \Lee_1(\R)$ for $1\le k\le n$  and  $\sum_{1\le k\le n} a_k = 1$. Define further $m\Cc:=\{af: |a|\le m \mbox{ and } f\in\Cc\}$ for each $m\ge 1$ and $\Dc:=\bigcup_{m\ge 1}m\Cc$. We denote by $\overline{\Dc}$ the $\tau-$closure of $\Dc$ in $\Lee(\R^d)$. Similarly, we define by $m\overline{\Cc}$ the $\tau-$closure of $m\Cc$. Then one has the following proposition.
\begin{proposition}\label{prop:closure}
{\rm (i)}  $\Dc$ is dense in $\Lee(\R^d)$ with respect to $\tau$, i.e. $\overline{\Dc}=\Lee(\mathbb{R}^d)$.

\vspace{1mm}

\noindent {\rm (ii)} Further,   $\Lee(\mathbb{R}^d) = \bigcup_{m\ge 1} m\overline{\Cc}$ holds. 
\end{proposition}
\begin{proof}
{\rm (i)} If $\overline{\Dc}\neq \Lee(\mathbb{R}^d)$, then by the Hahn–Banach theorem, see e.g. Corollary 3.15 in Chapter 4 of \cite{Conway}, there exists a non-zero $\tau-$continuous linear function $T:\Lee(\R^d)\to\R$ such that $T(f) = 0$ for all $f\in \Dc$, where 
\b* 
T(f)&=&\int_{\mathbb{R}^d} \big[f(\bx)u(\bx)+
\nabla f(\bx)\cdot{w(\bx)}\big]d\bx,\quad \mbox{for some } (u, w)\in L^1\big(\R^d, (1+|\bx|)d\bx\big)\times  L^1(\R^d)^d.
\e* 
Take the bump function $\varphi:\R^d\to\R_+$ defined by
\b* 
\varphi(\bx)&:=& \begin{cases} 
      c\exp\big(1/(|\bx|^2-1)\big) & \mbox{if } |\bx|\le 1 \\
      0 & \mbox{otherwise}, 
   \end{cases}
\e* 
where $c>0$ is chosen such that $\int_{\R^d}\varphi(\bx)d\bx=1$.
Define further the sequence $(\varphi_t)_{t>0}$ with $\varphi_t(\bx):=\varphi(\bx/t)/t^d$. For every $f\in\Dc$, one has the convolution $\varphi_t\ast f\in\Dc$ and thus
\be\label{eq:fourier}
0 = \int_{\R^d} \big[u(\bx) (\varphi_t\ast f)(\bx)+ w(\bx) \cdot{\nabla (\varphi_t\ast f)(\bx)}\big] d\bx 
=  \int_{\R^d} \big[u\ast{\varphi_t}(\bx) - {\rm div}(w\ast{\varphi_t})(\bx)\big]f(\bx)d\bx, 
\ee 
where the integration by parts can be applied thanks  to the convolution. Taking respectively $f(\bx)=\cos(2\pi \bz\cdot \bx)$ and $f(\bx)=\sin(2\pi \bz\cdot \bx)$ for $\bz\in\R^d$, one deduces that the Fourier transform of $u\ast{\varphi_t} -{\rm div}(w\ast{\varphi_t})$ is identically  equal to zero, i.e.
\b* 
 \int_{\R^d} \big[u\ast{\varphi_t}(\bx) - {\rm div}(w\ast{\varphi_t})(\bx)\big]e^{-2i\pi \bz\cdot \bx}d\bx &=& 0,\quad \mbox{for all } \bz\in\R^d,  
\e* 
and thus $u\ast{\varphi_t} -{\rm div}(w\ast{\varphi_t})\equiv 0$. Therefore,  \eqref{eq:fourier} holds for any $f\in \Lee(\mathbb{R}^d)$. Further, 
\b* 
&& \left|\int_{\R^d} \Big[u(\bx) \big((\varphi_t\ast f)(\bx) -f (\bx)\big)+ w(\bx) \cdot\big({\nabla (\varphi_t\ast f)(\bx)-\nabla f(\bx)\big)}\Big] d\bx \right| \\
&\le & \int_{\R^d} |u(\bx)| \left(\int_{\R^d}\varphi(\by)\big |f(\bx-t\by)-f(\bx)\big |d\by\right)d\bx + \int_{\R^d} |w(\bx)| \left|\int_{\R^d}\varphi(\by)\nabla f(\bx-t\by)d\by -\nabla f(\bx)\right|d\bx.
\e* 
 Using the dominated convergence theorem, and the Lebesgue-Besicovitch differentiation theorem, see e.g. page 43 of \cite{Evans}, for the second term, one has $T(f)=0$ for any $f \in \Lee(\mathbb{R}^d)$, contradicting the fact that $T$ is not null.

\vspace{1mm}

\noindent {\rm (ii)} Let $f\in \Lee(\mathbb{R}^d)$. Then there exists a net $(f^{\lambda})_{\lambda}\sb \Dc$ such that $f^{\lambda}$ converges to $f$ under $\tau$. Hence, the continuous linear functions $F_{\lambda}:L^1(\R^d)^d\to\R$ defined by 
\b* 
F_{\lambda}(w):=\int_{\mathbb{R}^d}  \nabla f^{\lambda}(\bx)\cdot w(\bx) d\bx
\e* 
are pointwise bounded. By the uniform boundedness principle, it holds 
$\sup_{\lambda} \|f^{\lambda}\|_{\lip}=\sup_{\lambda} \|\nabla f^{\lambda}\|_{\infty} = \sup_{\lambda} \|F_{\lambda}\|<\infty$. Thus $f\in m\overline{\Cc}$ for any   
$m\ge \sup_{\lambda} \|f^{\lambda}\|_{\lip}$.
\end{proof}
\begin{proposition}\label{prop:closed}
$\overline{\Cc}$ is closed with respect to the norm $\|\cdot\|_{\lip}$.
\end{proposition}
\begin{proof}
First, we show that the topology $\tau$ restricted to $\Lee_1(\R^d)\sb \Lee(\R^d)$ is metrizable. Since $L^1\big(\R^d, (1+|\bx|)d\bx\big)$ and $L^1(\R^d)^d$ are separable, we may take two dense subsets  $(u_i)_{i\ge 1}$ and $(w_j)_{j\ge 1}$ and define $\rho_{u_i,w_j}: \Lee(\R^d)\times \Lee(\R^d)\to\R_+$ by
\b* 
\rho_{u_i,w_j}(f,g) &:=& \left|\int_{\mathbb{R}^d} \big[(f-g)(\bx)u_i(\bx)+
\nabla (f-g)(\bx)\cdot{w_j(\bx)}\big]d\bx\right|.
\e* 
Then by a straightforward verification,  the distance $\rho: \Lee(\R^d)\times \Lee(\R^d)\to\R_+$ given by
\b* 
\rho(f,g)&:=&\sum_{i,j\ge 1}\frac{1}{2^{i+j}}\frac{\rho_{u_i,w_j}(f,g)}{1+\rho_{u_i,w_j}(f,g)}
\e* 
is consistent with the topology $\tau$ restricted on $\Lee_1(\R^d)$. Second, we prove $\overline\Cc\subset\Lee_1(\R^d)$. For any $f\in \overline\Cc$, one has a sequence $(f_n)_{n\ge 1}\subset \Cc$ converging to $f$ as $\tau$ is metrizable on $\Cc$. For any $w\in L^1(\R^d)^d$, it follows that
\b* 
\lim_{n\to\infty} \int_{\R^d}\nabla f_n(\bx)\cdot w(\bx)d\bx &=& \int_{\R^d}\nabla f(\bx)\cdot w(\bx)d\bx,
\e* 
which means that 
    $\nabla f_n$ converges to  $\nabla f$  
under the weak* topology of  $L^{\infty}(\R^d)^d$. Note also,  in view of the Banach-Alaoglu theorem,   $(\nabla f_n)_{n\ge 1}$ belongs to the unit ball $B_1^\infty\subset  L^\infty(\mathbb R^d)^{d}$ which is relatively compact with respect to the weak* topology. Then the uniqueness of the weak* limit yields $\nabla f\in B^\infty_1$, and thus $f\in \Lee_1(\R^d)$. Hence $\overline{\Cc}\subset  \Lee_1(\R^d)$. 

Let $f$ be in the closure of $\overline{\Cc}$ with respect to $\|\cdot\|_{\lip}$. Let $(f^n)_{n\ge 1}\sb\overline{\Cc}$ satisfying $\lim_{n\to\infty}\| \nabla (f^n- f)\|_\infty=\lim_{n\to\infty}\| f^n- f\|_{\lip}=0$, which implies in particular $ \lim_{n\to\infty}\rho(f^n,f)=0$ as $|f^n(\bx)-f(\bx)|\le \| f^n- f\|_{\lip}|\bx|$ for all $\bx\in\R^d$.   
For each $n\ge 1$, since $f^n \in \overline{\Cc}$, there exists $g^n\in\Cc$ such that $\rho(g^n,f^n)\le 1/n$. Then $\lim_{n\to\infty}\rho(g^n,f) \le \lim_{n\to\infty}\big(\rho(g^n,f^n) + \rho(f^n,f)\big)=0$, which concludes the proof.
\end{proof}
\begin{theorem}\label{prop:inclusion}
There exists a constant $C_d>0$ such that $\Lee_1(\R^d)\sb  C_d\overline{\Cc}$. 
\end{theorem}
\begin{proof}
In view of Propositions \ref{prop:closure} and \ref{prop:closed},  $\Lee(\R^d)= \bigcup_{m\ge 1} m\overline{\Cc}$ and  $m\overline{\Cc}$ is closed with respect to $\|\cdot\|_{\lip}$ for each $m\ge 1$. Now it follows from Baire's theorem that there must  exist $m^*\ge 1$ such that $m^*\overline{\Cc}$ has non-empty interior, i.e.  
\b* 
\big\{f\in \Lee(\R^d): \|f-f^*\|_{\lip}<\eps^*\big\} &\sb& m^*\overline{\Cc}, \quad \mbox{for some } f^*\in \Lee(\R^d) \mbox{ and } \eps^*>0. 
\e* 
By the proof of Proposition  \ref{prop:closure}, there exists $m_0\ge 1$ such that $f^*\in m_0\overline{\Cc}$. Thus one has 
\b* 
\big\{f\in \Lee(\R^d): \|f\|_{\lip}<\eps^*\big\} &\sb& (m^*+m_0)\overline{\Cc}.
\e* 
and further $\Lee_1(\R^d)\sb  C_d\overline{\Cc}$ with $C_d:=2(m^*+m_0)/\eps^*$.
\end{proof}
\begin{remark}
Proposition \ref{prop:closure} and Theorem \ref{prop:inclusion} form the basis to prove {\rm (ii)} of Theorem \ref{thm:equivalence}.   It is worth pointing out that, with a suitable adaptation, the proof of the universal approximation theorem allows to show the density of $\bigcup_{m\ge 1}m\Cc$ in $\Lee(\R^d)$ with respect to $\tau$. However, we cannot prove something similar to  Proposition \ref{prop:closure} {\rm (ii)} under the uniform norm. Therefore, Theorem \ref{prop:inclusion}, which proves that for the convex absorbing subset $\Lee_1(\R^d)$ there exists $m$ large enough such that $m\Cc$ is also dense in $\Lee_1(\R^d)$, needs additional arguments.
\end{remark}
Now we have all the ingredients  to prove Theorem \ref{thm:equivalence} {\rm (ii)}. 
\begin{proof}[of Theorem \ref{thm:equivalence} {\rm (ii)}]
We assume first $ \mu,\nu $ have densities $ u, v\in L^1\big(\R^d, (1+|\bx|)d\bx\big)$ with respect to the Lebesgue measure. Recall Kantorovich's dual formulation \eqref{def:dual}, and one has
\b* 
\Wc_{1}(\mu,\nu)&=& \sup_{f \in \Lee_1(\R^d)} \int_{\mathbb{R}^d} f(\bx) \big(u(\bx)-v(\bx)\big)d\bx \\ 
&\le& \sup_{f\in C_d\overline{\Cc}}  \int_{\mathbb{R}^d} f(\bx) \big(u(\bx)-v(\bx)\big)d\bx \\
&=& C_d\sup_{f\in\Cc}  \int_{\mathbb{R}^d} f(\bx) \big(u(\bx)-v(\bx)\big)d\bx,
\e*
where $C_d$ is the constant in Theorem  \ref{prop:inclusion}. As $\Cc$ is the collection of convex combinations of functions $f(\bv \cdot \bx)$ with $\bv\in\S^{d-1}$ and $f \in \Lee_1(\R)$, it follows that
\b* 
\Wc_{1}(\mu,\nu)
&\le & C_d\sup_{f\in\Cc}  \int_{\mathbb{R}^d} f(\bx) \big(u(\bx)-v(\bx)\big)d\bx\\
&=&C_d\sup_{f\in \Lee_1(\R), \bv\in\S^{d-1}} \int_{\mathbb{R}^d} f(\bv\cdot \bx) \big(u(\bx)-v(\bx)\big)d\bx \\
&=& C_d\sup_{\bv\in\S^{d-1}} \Wc_1(\mu_{\bv},\nu_{\bv})= C_d\mW_1(\mu,\nu).
\e* 
Hence, \eqref{eq:equiv} is established in view of Lemma \ref{lem:density}.
\end{proof}
\begin{lemma}\label{lem:density}
The subset of probability measures admitting a density is dense in $\Pc_1(\R^d)$  under $\Wc_1$ and $\mW_1$.
\end{lemma}
\begin{proof}
Fix an arbitrary $\mu\in\Pc_1(\R^d)$ and take the density function $\varphi:\R^d\to\R_+$ given by
\b* 
\varphi(\bx)&:=& \begin{cases} 
      c\exp\big(1/(|\bx|^2-1)\big) & \mbox{if } |\bx|\le 1 \\
      0 & \mbox{otherwise}, 
   \end{cases}
\e* 
where $c>0$ is chosen such that $\int_{\R^d}\varphi(\bx)d\bx=1$. Define the sequence of convolutions of measures $(\mu_t)_{t>0}$ by $\mu_t:=\mu\ast \nu_t$, where the probability measure $\nu_t$ is identified by its density function $\varphi_t(\bx):=\varphi(\bx/t)/t^d$. By construction $\mu_t$ admits a density, and it remains to estimate $\Wc_1(\mu_t,\mu)$ according to
\eqref{def:dual}. For each $f\in\Lee_1(\R^d)$, it holds
 \b* 
 \int_{\R^d}f(\bx)\mu_t(d\bx) - \int_{\R^d}f(\bx)\mu(d\bx) &:=&  \int_{\R^d\times\R^d}f(\bx+\by)\mu(d\bx)\varphi_t(\by)d\by - \int_{\R^d}f(\bx)\mu(d\bx) \\
  &=&  \int_{\R^d\times\R^d}\big(f(\bx+t\bz)-f(\bx)\big)\mu(d\bx)\varphi(\bz)d\bz \\
  &\le& t\int_{\R^d}\bz\varphi(\bz)d\bz \le t,
  \e*
  which implies $\Wc_1(\mu_t,\mu)\le t$ and thus the desired density under $\Wc_1$. The density under $\mW_1$ follows immediately from the inequality 
  $\mW_1(\mu_t,\mu)\le \Wc_1(\mu_t,\mu)$.
  \end{proof}
 

\subsection{Proof of Theorem~\ref{thm:equivalence} {\rm (iii)}}\label{ssec:proof3}  

\noindent To complete the proof of Theorem \ref{thm:equivalence}, we need an auxiliary result. Given a generic probability measure $\mu\in \Pc_p(\R^d)$,  $\hmu^n$ is said to be its empirical measure (of order $n$) if 
\b* 
\hmu^n(d\bx) &=& \frac{1}{n}\sum_{k=1}^n\delta_{X^k}(d\bx), 
\e* 
where $(X^k)_{k\ge 1}$ is a sequence of i.i.d. random variables such that $X^k\sim \mu$ for all $k\ge 1$. By  Theorem 1 of \cite{FG}, there exist $C\big(p,d, M_p(\mu)\big)>0$ and $\chi_{p,d}:\N\to\R_+$ such that
\b* 
\E\big[\Wc_1(\mu,\hmu^n)\big] &\le&  C\big(p,d, M_p(\mu)\big)\chi_{p,d}(n),\quad \mbox{for all } n\ge 1.
\e* 
The function $\chi_{p,d}$ is specified in \cite{FG}, while we need to refer to \cite{GO} for the explicit expression of $C\big(p,d, M_p(\mu)\big)$. For $d=1$ and $p=3$, one has  
\be \label{eq:estimate1}
\E\big[\Wc_1(\mu,\hmu^n)\big] &\le&  \frac{17280\big(M_3(\mu)^3+1\big)}{\sqrt{n}} ,\quad \mbox{for all } n\ge 1.
\ee  
\begin{proof}[of Theorem \ref{thm:equivalence} {\rm (iii)}]
We argue by contradiction. Assume that there exists $C>0$ such that $\Wc_p(\mu,\nu)\le C\sW_p(\mu,\nu)$ for all $\mu,\nu\in\Pc_p(\R^d)$.
Let $\ell^d\in\Pc_p(\R^d)$ be the Lebesgue measure concentrated on $[0,1]^d$. Let $\big(G^n\equiv (G^n_1,\ldots, G^n_d)\big)_{n\ge 1}$ be a sequence of i.i.d. random variables distributed according to  $\ell^d$. Define $\hmu^n$ to be the empirical measure given by
\b* 
\hmu^n(d\bx) &:=& \frac{1}{n}\sum_{k=1}^n \delta_{G^k}(d\bx).
\e* 
For any two sequences $(a_n)_{n\ge 1}, (b_n)_{n\ge 1}\sb\R_+$, we say $a_n \approx b_n$ if there exists a constant $c>0$ such that 
\b* 
\frac{a_n}{c} \le b_n \le ca_n,\quad \mbox{for all } n\ge 1.
\e* 
Then it follows by \cite{AGT} that
\begin{equation}\label{eq:convergence}
    \E\big[\Wc_1(\ell^d,\hmu^n)\big]  \approx  \begin{cases} 
      \big(\log(n)/n\big)^{1/2} & \mbox{if } d=2 \\
      n^{-1/d} & \mbox{if } d\ge 3.
   \end{cases}
\end{equation}
On the other hand, one has by assumption
\b* 
\E\big[\Wc_1(\ell^d,\hmu^n)\big] \le    C\E\big[\sW_1(\ell^d,\hmu^n)\big] = C\oint_{\S^{d-1}}\E\big[\Wc_1(\ell^d_{\bv},\hmu^n_{\bv})\big]d\bv,
\e* 
where we recall that  $\ell^d_{\bv}$ and  $\hmu^n_{\bv}$ are the projections of $\ell^d$ and $\hmu^n$ along the direction $\bv=(v_1,\ldots, v_d)$. Substituting $\ell^d_{\bv}$ and $\hmu^n_{\bv}$ into \eqref{eq:estimate1}, one has 
\b* 
\E\big[\Wc_1(\ell^d_{\bv},\hmu^n_{\bv})\big] \le \frac{17280\big(M_3(\ell^d_{\bv})^3+1\big)}{\sqrt{n}} \le \frac{17280(d^3+1)}{\sqrt{n}},\quad \mbox{for all } \bv\in\S^{d-1}, 
\e* 
where the second inequality follows from $M_3(\ell^d_{\bv})^3 \le M_3(\ell^d)^3 \le d^3$. This yields
\b* 
\E\big[\Wc_1(\ell^d,\hmu^n)\big] &\le& \frac{17280CA_d\big(M_3(\ell^d_{\bv})^3+1\big)}{\sqrt{n}},
\e*  
and further, combined with \eqref{eq:convergence}, yields a contradiction for $d\ge 2$ and concludes the proof.
 \end{proof}

\appendix 

\section{}

\noindent We start by recalling some elementary ingredients from functional analysis. Given a topological vector space $\Ec$, we denote by  $\Ec^*$ its dual space in separating duality via a bilinear form. The $\mbox{weak}^*$ convergence, denoted by $w^*$, is the convergence on $\Ec^*$ induced by the elements of $\Ec$, i.e. $(e^*_n)_{n\ge 1}\sb \Ec^*$ converges to $e^*\in\Ec$ under $w^*$ if and only if
\b* 
\lim_{n\to\infty}\langle e^*_n, e \rangle &=& \langle e^*, e\rangle,\quad \mbox{for all } e\in\Ec.
\e* 
Endowed with $w^*$, the dual space of $\Ec^*$ is isometric to $\Ec$. In the following, we set  $\Ec=L^1(\R^d)^{d+1}$ and $\Ec^*=L^{\infty}(\R^d)^{d+1}$ which are respectively  the $(d+1)-$product of $L^{1}(\R^d)$ and $L^{\infty}(\R^d)$. 

\begin{proof}[of Lemma \ref{lem:dual}]
Note that 
 $\Lee(\R^d)$ embeds into the space $L^\infty(\mathbb{R}^d)^{d+1}$ via the map $\Lc: \Lee(\R^d)\to L^\infty(\mathbb{R}^d)^{d+1}$ defined by
 \b* 
 \Lc(f)&:=& \left(\frac{f}{1+|\bx|},\nabla f:=\big(\partial_1 f,\ldots,\partial_d f\big)\right).
 \e* 
For each function $T:\Lee(\R^d)\to\R$, define $T\circ \Lc^{-1}: \Lc\big(\Lee(\R^d)\big) \to\R$ by 
\b* 
\Tc\big(\Lc(f)\big) &:=& T(f).
\e* 
Then by definition, $T$ is  $\tau-$continuous linear if and only if $T\circ \Lc^{-1}$ is $w^*-$continuous linear. Following the arguments of Lemma 2.2 in \cite{GL}, $\Lc\big(\Lee(\R^d)\big)\sb L^\infty(\mathbb{R}^d)^{d+1}$,
 presented by weak equations
\b* 
\Lc\big(\Lee(\R^d)\big) &=&\Big\{ (g ,w)\in L^\infty(\mathbb{R}^d)^{d+1}:~  \partial_i\big((1+|\bx|)g \big)= w_i, \;  \partial_iw_j=\partial_jw_i, \mbox{ for } 1\le i,j\le d \Big\}
\e* 
is a $w^*-$closed subspace. Then the dual spaces of $\Lc\big(\Lee(\R^d)\big)$ is included in the dual space of  $L^1(\R^d)^{d+1}$, see e.g. page 129 in  \cite{Conway}. The proof is fulfilled by the fact that the elements $(g,w)$ of $L^1(\mathbb{R}^d)^{d+1}$ represent all $\tau-$continuous linear functions on  $\Lc\big(\Lee(\R^d)\big)$ via
\b* 
\Lc\big(\Lee(\R^d)\big)\ni \Lc(f) &\mapsto& \int_{\R^d}g(\bx)\frac{f(\bx)}{1+|\bx|}d\bx+\int_{\R^d} \nabla f(\bx) \cdot w(\bx)d\bx\in\R. 
\e* 
\end{proof}

\bibliographystyle{abbrv}
\bibliography{annot}

\affiliationone{
   Erhan Bayraktar, 
   University of Michigan\\
   \email{erhan@umich.edu}}
\affiliationtwo{
   Gaoyue Guo, 
   CentraleSup\'elec\\
   \email{gaoyue.guo@centralesupelec.fr}}

\end{document}